\documentclass[a4paper,12pt]{article}
\usepackage{amsthm}
\usepackage{mathrsfs}

\usepackage{amssymb}
\usepackage{epsfig}
\usepackage{amsfonts}
\usepackage{amsmath}
\usepackage{euscript}
\usepackage{amscd}
\usepackage{amsthm}
\usepackage{theoremref}
\DeclareMathAlphabet{\mathpzc}{OT1}{pzc}{m}{it}

\newtheorem{thm}{Theorem}[section]
\newtheorem{lem}[thm]{Lemma}
 
\newtheorem{cor}[thm]{Corollary}
\newtheorem{rem}[thm]{Remark}
\newtheorem{ex}[thm]{Example}

\newtheorem{definition}{Definition}[section]

\newcommand{\m}{\mathpzc{m}}
\newcommand{\p}{\mathpzc{p}}

\newcommand{\bZ}{\mathbb Z}

\newcommand{\Spec}{\operatorname{Spec}}\newcommand{\hgt}{\operatorname{ht}}
\newcommand{\td}{\operatorname{tr.deg}}
\newcommand{\LND}{\operatorname{LND}}

\title{Some results on homogeneous locally nilpotent $R$-derivations on $R[X,Y,Z]$}
\author{Parnashree Ghosh \\
	{\small{\it  Theoretical Statistics and Mathematics  Unit, Indian Statistical Institute,}}\\ 
	{\small{\it 203 B.T.Road, Kolkata-700108, India}}\\
	{\small{\it e-mail : ghoshparnashree@gmail.com, parnashree$\_$r@isical.ac.in}}\\
}

\begin{document}
\date{}
\maketitle

\abstract{
	\noindent
	Let $k$ be a field of characteristic zero and $R$ a $k$-algebra. In this note we study homogeneous $R$-lnds $D$ on $R[X,Y,Z]$ with respect to the standard weights $(1,1,1)$. We show that when $R$ is a PID, $rank(D)$ can be at most $2$ if $\deg(D) \leqslant 3$. As a consequence we obtain a certain class of homogeneous lnds on $k^{[4]}$ whose kernels are $k^{[3]}$. Further when $R$ is a Dedekind domain, we 
	show that $\ker(D)$ is generated by at most $3$ elements
	 if $\deg(D) \leqslant 3$. 
	}

\smallskip

\noindent
{\small {{\bf Keywords}. Polynomial Rings, Homogeneous Locally Nilpotent Derivation.}

}

\section{Introduction}

Throughout this note, $k$ will denote a field of characteristic zero.
For an integral domain $B$ and a non-zero element $b \in B$, the notation $B_b$ will denote the localisation of $B$ with respect to the multiplicatively closed subset $\{b^n \mid n \in \mathbb{Z}_{\ge 0}\}$ of $B$.
Let $S \subset R$ be an integral domains containing $k$. 
Then $\LND(R)$ will denote the set of all locally nilpotent derivation(s) (lnd(s)) on $R$ and $\LND_S(R)$ will denote the set of all locally nilpotent $S$-derivations $D$ ($S$-lnd, i.e., $D|_S=0$)  on $R$.  
By $R^{[n]}$ we denote a polynomial ring in $n\, (\geqslant 1)$ variables over $R$ and $R^{*}$ denotes the multiplicative group of units in $R$.
 For a $\mathbb{Z}$-graded integral domain $B$ and a homogeneous lnd $D$ on $B$, $\deg(D)$ will denote its degree with respect to the grading on $B$ (see \thref{def}(v)). 
 Unless specified, by homogeneous $R$-lnd $D$ on $R[X,Y,Z]$, we mean $D$ is homogeneous with respect to the weights $(1,1,1)$.
For an lnd $D \in \LND_R(R^{[n]})$, the rank of $D$, denoted by $rank(D)$, is defined as the smallest positive integer $r$ such that there exists a coordinate system $\{V_{1}, \cdots, V_{n}\}$ in $R^{[n]}$  for which $DV_i \neq 0$ for $i \leqslant r$, and $DV_i=0$ for $i>r$. 

To understand an lnd $D$ on an integral domain $B$, it is very important to know its kernel ($\ker(D)$), and when $B=R^{[n]}$ and $D \in \LND_R(B)$, then $rank(D)$ plays a crucial role to determine the structure of $\ker(D)$.  
%
%
%

\smallskip
We now state a result of Freudenburg on the rank of a homogeneous lnd on $k^{[3]}$ (\cite[Corollary 2]{GF}).
\begin{thm}\thlabel{thm1}
	Let $k$ be an algebraically closed field of characteristic zero, and $D$ be a homogeneous locally nilpotent derivation on $k[X,Y,Z] (=k^{[3]})$. If $\deg(D) \leqslant 3$, then $rank(D) \leqslant 2$.
\end{thm}
\noindent
In section 2 of this note we show that \thref{thm1} can be extended to any field of characteristic zero (\thref{thma}).
In \cite{GF}, a crucial step to prove \thref{thm1} is the following result (\cite[Proposition 3]{GF}): 

\begin{thm}\thlabel{prop1}
		Let $k$ be an algebraically closed field, and $D$ a homogeneous locally nilpotent derivation on $k[X,Y,Z]$ with respect to the weights $(1,1,1)$. If $rank(D)=3$, then one of the following holds. 
	
	\begin{itemize}
		\item [\rm(a)] There exists a system of coordinates $\{U,V,W\}$, linear in $X,Y,Z$ such that $$\deg_{D}(U)<\deg_{D}(V)<\deg_{D}(W).$$
		
		\item[\rm(b)] Every linear term in $k[X,Y,Z]$ has the same $\deg_{D}$-value.
	\end{itemize}
\end{thm}
\noindent
In section 3, we show that the possibility (b) in \thref{prop1} does not hold. In particular we prove the following result  (\thref{li}).

\medskip
\noindent
{\bf Theorem.}
	Let $D$ be a homogeneous locally nilpotent derivation with respect to the standard weights $(1,1,1)$ on $B:=k[X,Y,Z]$ such that $rank(D) >1$. Then there exists a linear system of coordinates $\{L_{1}, L_{2}, L_{3}\}$ in $k[X,Y,Z]$ such that $$\deg_{D}(L_{1})< \deg_{D}(L_{2})< \deg_{D}(L_{3}).$$

 Using the above result, we give an alternative proof of \thref{thm1} in the Appendix (\thref{ra}).


\smallskip
In view of \thref{thm1}, we further ask the following question:

\medskip

\noindent 
\textbf{Question 1.} 
Let $R$ be a normal domain containing a field of characteristic zero, and $D$ be a homogeneous $R$-lnd on $R[X,Y,Z]$ such that $\deg(D) \leqslant 3$. 
Does this imply that $rank(D) < 3$?

\medskip
\noindent
In Section 4, we answer Question 1 affirmatively when $R$ is a PID and hence $\ker(D)=R^{[2]}$ (\thref{rpid}).
By the work of Rentschler (\cite{ren}) and Miyanishi (\cite{miya}) we know that for any lnd $D$ on $k^{[n]}$, $\ker(D)= k^{[n-1]}$, when $n \leqslant 3$. But such a result does not hold for all lnds over $k^{[4]}$. However, for an lnd $D \in \LND(k^{[4]})$ if $rank(D)<4$ and $\ker(D)$ is regular, then Bhatwadekar, Gupta and Lokhande have shown that $\ker(D)=k^{[3]}$ (\cite[Theorem 3.5]{bgl}). In \cite[Corollary 3.8]{nice}, Dasgupta and Gupta have shown that if $D$ is a nice lnd on $k^{[4]}$ of rank at most $3$, then $\ker (D)=k^{[3]}$. 
As an immediate consequence of \thref{rpid} of this note, we get the following result which exhibits another class of lnds over $k^{[4]}$ whose kernel is $k^{[3]}$ (\thref{cpid}). 

\medskip
\noindent
{\bf Theorem A.}
Let $D$ be a homogeneous locally nilpotent derivation on $k[X_{1},\dots,X_{4}]$ of degree at most $3$ with respect to the weights $(0,1,1,1)$ such that $DX_{1}=0$. Then $\ker(D)=k^{[3]}$ and $rank(D)<3$.

\medskip

 We further investigate Question 1 for Dedekind domains and higher dimensional UFDs. 
 We prove the following result which answers Question 1 affirmatively for Dedekind domains, under a sufficient condition (\thref{dd}):

\medskip
\noindent
{\bf Theorem B.}
		Let $R$ be a Dedekind domain containing $\mathbb{Q}$ and $D$ be a homogeneous locally nilpotent $R$-derivation on $R[X,Y,Z]$ such that  $\deg(D) \leqslant 3$. 
	    Then the following hold.
	
	\smallskip
	\noindent
	{\rm (i)}
	$\ker(D)$ is generated by at most $3$ elements. 
	
	\smallskip
	\noindent
	{\rm (ii)}
	If there exists a homogeneous system of generators of $\ker(D)$ among which exactly one term is linear, then that is a coordinate in $R[X,Y,Z]$ and hence $rank(D)<3$.

\medskip
\noindent
We also show that if condition (ii) in Theorem B is not satisfied, then $D$ can have full rank (\thref{exa 2}). 
 In \cite[Proposition 4.13]{bhd}, Bhatwadekar and Daigle have shown that for every  $R$-lnd $D$ on $R^{[3]}$, $\ker(D)$ is a finitely generated $R$-algebra, when $R$ is a Dedekind domain containing $\mathbb{Q}$. Further Daigle and Freudenburg have shown existence of a triangular lnd $D_n$ on $k^{[4]}$ whose kernel can not be generated by less than $n$ elements for every $n \in \mathbb{N}$, $n \geqslant 3$ (\cite{df}). Therefore, in general, over a Dedekind domain $R$, for an $R$-lnd $D$ on $R[X,Y,Z]$ there is no specific upper bound on minimum number of generators of $\ker(D)$. 
 But with the additional hypothesis that $D$ is homogeneous with $\deg(D) \leqslant 3$, Theorem B indeed gives an upper bound on minimum number of generators of $\ker(D)$.
 Also we give examples of higher dimensional UFDs for which the answer to Question 1 is negative (\thref{exa4}).       

In the next section we discuss some basic concepts on locally nilpotent derivations and some earlier results which will be used in this note.

\section{Preliminaries}
 We first recall some notation and definitions on locally nilpotent derivations (cf. \cite{GFB}). 

\begin{definition}\thlabel{def}
	\rm{Let $B$ be an integral domain containing $k$, $D$ a non-trivial lnd on $B$, and $A=\ker(D)$.}
	\begin{enumerate}
		\item[\rm{(i)}] \rm{ An element $r \in B$ is called a {\it local slice} of $D$, if $Dr \in ker D \setminus \{0\}$.}
		
		\item[\rm{(ii)}] \rm{$D$ is said to be {\it irreducible} if the ideal $\left(D(B)\right)$ is not contained in any proper principal ideal of $B$.
		}
		
		\item[\rm(iii)] \rm{$D$ defines a degree function $\mu:=\deg_{D}$ on $B$ such that  $\deg_{D}(0)=-\infty$ and for every nonzero $b \in B$ 
		$$
		\mu(b)\, (= \deg_{D}(b))=max\{n \in \mathbb{N}\,|\, D^n(b) \neq 0\}.
		$$}
		
		\item[\rm{(iv)}] \rm{Let $\mu$ be the degree function on $B$ induced by $D$. For a non-negative integer $n$, we define the {\it $n$-th degree $A$-module} with respect to $\mu$, as follows : 
			$$	
			\mathscr{F}_{n}=\{b \in B \mid \mu(b) \leqslant n\}.
			$$ }

		
		\item[\rm{(v)}] \rm{Let $G$ be a totally ordered Abelian group and $B$ a $G$-graded ring such that $B=\bigoplus_{i \in G}B_i$ is the $G$-graded structure on $B$. Then $D$ is said to a {\it homogeneous derivation} on $B$, if there exists some $d \in G$, such that $DB_i \subseteq B_{i+d}$ for every $i \in G$, and $\deg_G(D):=d$ is said to be the {\it degree} of $D$. 
		If $G=\mathbb{Z}$, then $\deg_{\mathbb{Z}}(D)$ is denoted by $\deg(D)$.} 
		
		
		\item[\rm{(vi)}] \rm{Let $B=k[X_1,\ldots,X_n]$ and $f_1,\ldots,f_{n-1} \in B$. For $\underline{f}:=(f_1,\ldots,f_{n-1})$, the {\it Jacobian derivation} $\Delta_{\underline{f}}$ on $B$ is defined as follows: 
			For every $g \in B$,
			$$
			\Delta_{\underline{f}}(g):=\frac{\partial(f_1,\ldots,f_{n-1},g)}{\partial(X_1,\ldots,X_n)}.
			$$
			}	 
	\end{enumerate}
	
\end{definition}

We now quote an important property of locally nilpotent derivations (cf. \cite[Principle 11, pg-27]{GFB}).

\begin{thm}\thlabel{localslice}
	Let $B$ be an integral domain containing $k$ and $D$ a non-zero locally nipotent derivation on $B$. Then for any local slice $r$ of $D$, $$B_{Dr}=(\ker (D))_{Dr}[r]~ (= (\ker (D)_{Dr})^{[1]}).$$ 
\end{thm}

The following lemma shows that in \thref{thm1} the assumption that ``$k$ is algebraically closed" can be dropped.

\begin{lem}\thlabel{thma}
	 Let $D$ be a homogeneous locally nilpotent derivation on $k[X,Y,Z]$ such that $\deg(D) \leqslant 3$. Then $rank(D) \leqslant 2$. 
\end{lem}
\begin{proof}
	Let $\overline{k}$ be an algebraic closure of $k$. 
	Note that $D$ extends to a homogeneous lnd $\overline{D}:=D \otimes_{k} \overline{k}$ on $\overline{k}[X,Y,Z]$ such that $\deg(\overline{D})=\deg(D) \leqslant 3$.
	 By \thref{thm1}, $rank(\overline{D})<3$, hence there exists a coordinate $V$ in $\overline{k}[X,Y,Z]$ such that $V\in \ker(\overline{D})$. 
	 Since $\overline{D}$ is homogeneous, we can assume that $V$ is linear in $X,Y,Z$. 
	 If $V \in k[X,Y,Z]$, then we are done. 
	 If not, then $V=\alpha_{1}X+\alpha_{2}Y+\alpha_{3}Z$ for some
	  $\alpha_1, \alpha_{2}, \alpha_{3} \in \overline{k}$ such that there exists some $i$, $1\leqslant i \leqslant 3$ such that $\alpha_i \in \overline{k} \setminus k$. 
	  Consider the finite field extension $\tilde{k}:=k(\alpha_{1},\alpha_{2}, \alpha_{3})$ of $k$, and let $\{\beta_1,\ldots,\beta_n\}$ be a basis for $\tilde{k}$ over $k$.  
	  Therefore, from the expression of $V$ it is clear that there exist linear terms $V_i \in k[X,Y,Z]$, such that 
	$$
	V= \sum_{i=1}^{n} \beta_i V_i.
	$$
	Since $V \in \ker(\overline{D})$ and $\beta_i$'s are linearly independent over $k$, it follows that $V_i \in \ker(D)$ for every $i, 1\leqslant i \leqslant n$. Hence $rank(D) \leqslant 2$.
\end{proof}
In the rest of this section we will recall some results which will be used later in this note.
The next lemma gives an important property of lnds (\cite[Section 1.2.4]{GFB}).

\begin{lem}\thlabel{loc}
	Let $R$ be an integral domain and $D \in \LND(R)$. Let $A=\ker(D)$ and $S$ be a multiplicatively closed subset of $A\setminus \{0\}$. Then $D$ will induce a locally nilpotent derivation $S^{-1}D$ on $S^{-1}R$. 
\end{lem}

Next we recall an important result of Wang. Before stating the result we give a few definitions as follows:

Let $B=k[X_{1}, \ldots , X_{m}]$. For an lnd $D$ on $B$, Wang (\cite[Section 3]{1})  defined a  new degree function $\overline{\mu}: B \to \bZ \cup \{-\infty\}$ with respect to $\mu := \deg_{D}$ such that 
\begin{equation}\label{mu}
\overline{\mu}(b)= max\left\{\mu(M_{j}) \mid \, b=\sum_{j=1}^{n}M_{j}\,\, \text{and}\, M_{j}\,\, \text{is a monomial in} \, X_{1}, \ldots , X_{m} \right\}.
\end{equation}
Now $B=k[X_{1}, \ldots, X_{m}]$ has a graded structure such that $B=\oplus_{i \in \mathbb{Z}} \overline{B}_{i}$ with $wt (X_{j})= \overline{\mu}(X_{j})=\mu(X_j)$, $1 \leqslant j \leqslant m$, where $\overline{B}_{i}$ is a vector space generated by
\begin{equation}\label{m}
\left\{b \in B \mid \,b\,\, \text{is a monomial in $X_{1},\ldots,X_{m}$ and}\,\, \overline{\mu}(b)=i \right\}.
\end{equation}
We now consider the associated graded integral domain of $B$ with respect to $\mu$, that is 
$$
gr_{\mu}B=\bigoplus_{i \geqslant 1} \frac{\mathscr{F}_{i}}{\mathscr{F}_{i-1}}.
$$
Since for every $b \in B$,  $\mu(b) \leqslant \overline{\mu}(b)$, we have a natural graded mapping 
\begin{equation}\label{phi}
\phi: B=\oplus_{i\in \mathbb{Z}} \overline{B}_{i} \rightarrow gr_{\mu}B ,
\end{equation}
such that $\ker(\phi)=\p$ is a $\overline{\mu}$-homogeneous prime ideal of $B$ such that a $\overline{\mu}$-homogeneous element $f\in \p$ if and only if $\mu(f)< \overline{\mu}(f)$.

\smallskip

 We now quote the following result of Wang (\cite[Lemma 4.3]{1}). 
\begin{lem}\thlabel{wa}
Let $B=k[X_{1}, X_{2},  X_{3}]$, $D$ a non-zero locally nilpotent derivation on $B$ and $\overline{\mu}$ the degree function on $B$ defined as in \eqref{mu}. If $\overline{\mu}(X_{i})= d_{i}$ and $D$ is $\overline{\mu}$-homogeneous, then ${\rm ht}(\ker(\phi))=2$  $(\phi \text{~as in~} (\ref{phi}))$ and for each $i,j$, $1 \leqslant i,j \leqslant 3$, there exists $\alpha_{ij} \in k^{*}$, such that $X_{i}^{d_{j}}-\alpha_{ij}X_{j}^{d_{i}} \in \ker(\phi)$.  
\end{lem}

The structure of an lnd on $k^{[3]}$ in terms of the Jacobian derivation is obtained from the following result of Daigle (\cite[Corollary 2.5]{daig}).

\begin{thm}\thlabel{dai}
	Let $B=k^{[n]}$ and $D \in \rm{LND}(B)$ such that $\ker(D)=k^{[n-1]}$. If $\{f_{1},\cdots,f_{n-1}\}$ be a set of algebraically independent elements in $B$ such that $\ker(D)=k[f_{1},\ldots,f_{n-1}]$, then $\Delta_{(f_{1},\ldots,f_{n-1})} \in \LND(B)$ and $D=a \Delta_{(f_{1},\ldots,f_{n-1})}$ for some $a \in \ker(D)$.
\end{thm}



The following theorem is due to Zurkowski (\cite{zur}), which gives the structure of the kernel of a homogeneous lnd on $k[X,Y,Z]$ with respect to a positive grading.

\begin{thm}\thlabel{zu}
Let $D$ be a non-zero homogeneous locally nilpotent derivation with respect to some positive grading $\omega$ on $k[X,Y,Z]$ and $A=\ker(D)$. Then there exist homogeneous polynomials $F, G$ with respect to the grading $\omega$ such that $A=k[F,G]$.
\end{thm}

Next we recall the following well-known result by Abhyankar, Eakin and Heinzer (\cite[Theorem 4.1]{AEH}).

\begin{thm}\thlabel{aeh}
Let $R$ and $S$ be \rm{UFDs} and $B= R^{[n]}$. 
Suppose that the transcendence degree of $S$ over $R$ is $1$,
and that $R \subseteq S \subseteq	B$. 
Then $S = R^{[1]}$.
\end{thm}

We now define a unimodular row over a commutative ring $R$ and quote a basic result on their `completability'. 

\begin{definition}
	\rm{Let $R$ be a commutative ring with unity. A row $\underline{r}:=(r_1,\ldots,r_n)$ of length $n$ with entries in $R$ is said to be unimodular, if there exist $s_1,\ldots,s_n \in R$ such that $\sum_{i=1}^{n} r_i s_i =1$. 
	A unimodular row $\underline{r}$ is said to be completable, if there exists $M \in \rm{GL}_n(R)$ such that $\underline{r}M =e_1$, where $e_1:=(1,0,\ldots,0)$. } 
\end{definition}
The following result is on completability of unimodular row over a commutative ring. For reference one can see \cite{bass}.  
\begin{thm}\thlabel{uni}
	Let $R$ be Noetherian ring of dimension $d$. Then every unimodular row on $R$ of length $n \geqslant d+2$ is completable.
\end{thm}

The following result was proved by Bass, Connell and Wright in \cite{bcw} and independently by Suslin in \cite{su}.

\begin{thm}\thlabel{bcw}
Let $A$ be a ring and $B$ a finitely presented $A$-algebra. Suppose that the $A_{\m}$-algebra $B_{\m}$ is isomorphic to the symmetric algebra
of some $A_{\m}$-module for each maximal ideal $\m$ of $A$. Then $B \cong {\rm{Sym}}_{A}(M)$ for some finitely presented $A$-module $M$.
\end{thm}

We conclude this section with Serre's Splitting Theorem (\cite[Theorem 7.1.8]{IR}).

\begin{thm} \thlabel{s}
Let $R$ be a Noetherian ring of finite Krull dimension $d$ and $P$ a
finitely generated projective $R$-module such that $rank(P)>d$. Then $P$ has a unimodular element. As a result, if $rank(P)=d+m \,(m \geqslant 1)$, then $P \cong R^{m} \oplus Q$, where $Q$ is a projective $R$-module of rank $d$.
\end{thm}

\section{Rank and kernel of homogeneous $R$-lnds on $R[X,Y,Z]$  }	
Throughout this section $D$ is a homogeneous $R$-lnd on $R[X,Y,Z]$ with respect to the weights $(1,1,1)$, where $R$ is a commutative Noetherian integral domain containing $\mathbb{Q}$. 
The main aim of this section is to investigate Question 1.
We begin this section with the following important lemma.

\begin{lem}\thlabel{lpid}
	Let $R$ be a {\rm PID} and $D$ be a  homogeneous locally nilpotent $R$-derivation on $R[X,Y,Z]$ with respect to the weights $(1,1,1)$. 
	Let $S=R \setminus \{0\}$ and $K=S^{-1}R$. 
	If $S^{-1}D \in \LND(K[X,Y,Z])$ denotes the natural extension of $D$, then $rank(S^{-1}D)=rank (D)$.
\end{lem}

\begin{proof}
	\smallskip
	\noindent	
	{\it \bf Case }{\bf 1}: If $rank(S^{-1}D)=3$, then clearly $rank(D)=3$.
	
	\smallskip
	\noindent
	{\it \bf Case }{\bf 2}: If $rank(S^{-1}D)=2$, then $rank(D) \geqslant 2$. 
	Since $rank(S^{-1}D)<3$ and $S^{-1}D$ is homogeneous lnd on $K[X,Y,Z]$, there exists a linear coordinate $L$ in $K[X,Y,Z]$ such that $L \in \ker(S^{-1}D)$. 
	By clearing denominators we can assume that $L=aX+bY+cZ$ where $a,b,c \in R$ and $\gcd(a,b,c)=1$.
	As $R$ is a PID, $(a, b, c)$ is a unimodular row over $R$ and hence completable. Therefore, $L$ is a coordinate in $R[X,Y,Z]$ and $L \in \ker(D)$. Therefore, $rank(D)<3$ and hence $rank(D)=2$.
	
	\smallskip
	\noindent
	{\it \bf Case }{\bf 3}: Let $rank(S^{-1}D)=1$. As shown in Case 2, we have $rank(D)<3$, and hence we get a linear system coordinates $\{U,V,W\}$ in $R[X,Y,Z]$ such that $U \in \ker(D)$. 
	Note that $K[X,Y,Z]=K[U,V,W]$ and $U \in \ker(S^{-1}D)$.
	Since $S^{-1}D$ is a homogeneous lnd and $rank(S^{-1}D)= 1$, $\ker(S^{-1}D)=K[U, N]$, where $N= \lambda V+ \mu W$ for some $\lambda, \mu \in K$. 
	Now clearing denominators we get $\tilde{N} :=a_{v}V+a_{w}W \in \ker(D)$ for some $a_v,a_w \in R$ and without loss of generality we can assume that $\gcd(a_{v},a_{w})=1$. 
	If $b_v,b_w \in R$ be such that $a_{v}b_{w}-a_{w}b_{v}=1$, then for
	$\tilde{P}=b_{v}V+b_{w}W$,
	$R[X,Y,Z]=R[U, \tilde{N}, \tilde{P}]$, and hence $rank(D)=1$.
\end{proof}

\begin{rem}
	\em{
		In \thref{lpid} none of the conditions that $D$ is homogeneous and $R$ is a PID can be dropped (see Examples \ref{expid} and \ref{r1} below). 
		
	}
\end{rem}

\begin{ex}\thlabel{expid}
	\em{For $R=k[t],$ we consider the $R$-lnd $D$ on $R[X,Y,Z]$ defined by
		$$ 
		DX=0 ,\,\, DY=X ,\,\, DZ=t.
		$$ 
		Note that $D$ is not homogeneous with respect to the weights $(1,1,1)$. Clearly, $\ker(D)=R[X, tY-XZ]$ and $rank(D)=2$. But for $S=R\setminus \{0\}$,  $\ker(S^{-1}D)=k(t)[X,tY-XZ]$, and hence $rank(S^{-1}D)=1$. Therefore, \thref{lpid} does not hold without the condition that $D$ is homogeneous.} 
\end{ex}


\begin{ex}\thlabel{r1}
	{\em
		Let  $B=R[X,Y,Z]$ where $R=\frac{\mathbb{R}[W_{1},W_{2}]}{(W_{1}^2+W_{2}^2-1)}.$ Let $w_{1}$ and $w_{2}$ denote the residue classes of $W_{1}$ and $W_{2}$ in $R$ respectively. Let $X_{1}=w_{1}X+(1-w_{2})Y \in R[X,Y,Z]$ and $X_{2}=(1+w_{2})X+w_{1}Y$.
		For $d \geqslant 0$ we define a homogeneous $R$-lnd $D$ of degree $d$ on $B$ as follows:
		\begin{align*}
			&DX=(1-w_{2})X_{1}^{d+1},\\
			&DY=-w_{1}X_{1}^{d+1},\\
			&DZ=(d+2)w_{1}Y^{d+1}, 
		\end{align*}
		We now show the following.
		
		\noindent
		\begin{enumerate}
			\item[(i)]  $DX_1=0$ and $DX_2=0$.
			
			\item[(ii)] $\deg_D (X)=\deg_D(Y)=1$ and $\deg_D(Z)=d+2$.

			\item[(iii)] Neither $X_{1}$ nor $X_2$ is a coordinate in $R[X,Y,Z]$.

			\item[(iv)] 	$A:=\ker(D) \neq R^{[2]}$.
			
			\item[(v)] $rank(D)=3$, and for $S=R \setminus \{0\}$, $rank(S^{-1}D)=2$.
		\end{enumerate} 
		Proofs of (i) and (ii) are easy to see. 
		
		\medskip
		\noindent
		(iii) We first recall that if $(r_1X+r_2Y+r_3Z)$ is a coordinate in $R[X,Y,Z]$, then $(r_1,r_2,r_3)R=R$. Now since $(w_{1}, 1-w_{2})$ and $(w_{1}, 1+w_{2})$ both are maximal ideals of $R$, we get that none of $X_{1}$ and $X_{2}$ is a coordinate in $R[X,Y,Z]$.
		
		\medskip
		\noindent
		(iv) Let $F:=(Y^{d+2}+X_{1}^{d+1}Z)$. Then 
		$$ R[X_{1}, X_{2}, F] \subseteq \ker(D)=A.$$ 
		Note that $(1+w_{2})X_{1}=w_{1}X_{2}$ and $(1-w_{2})X_{2}=w_{1}X_{1}$. Now 
		$R_{1+w_{2}}[X_{1},X_{2},F]=R_{1+w_{2}}[X_{2},F] \subseteq A_{1+w_{2}}$ such that both the rings have the same transcendence degree over $\mathbb{R}$, and also $R_{1+w_{2}}[X_{2},F]$ is factorially closed in $B_{1+w_2}$. Therefore, $$\ker(D_{1+w_{2}})=A_{1+w_{2}}=R_{1+w_{2}}[X_{2},F].$$ Similarly, $\ker(D_{1-w_{2}})=A_{1-w_{2}}=R_{1-w_{2}}[X_{1},F]$.
		Since the ideals $(1+w_{2})R$ and $(1-w_{2})R$ are comaximal, for every maximal ideal $\m$ of $R$, we have $A_{\m}=R_{\m}[X_{1},X_{2},F]$. Therefore, by the local-global principle, we have $A=R[X_{1}, X_{2}, F] $. Now if we put $U= \frac{X_{1}}{w_{1}}$, then $X_{2}=(1+w_{2})U$. Then we have 
		$$
		A=R[w_{1}U, (1+w_{2})U, F]\cong {\rm{Sym}}_{R}(I \oplus R)
		$$ 
		where $I=(w_{1}, 1+w_{2})R$ is a rank one projective $R$-module. Since the ideal $(w_{1}, 1+w_{2})R$ is not principal, it is not free. As rank 1 projective modules are cancellative, $I$ can not be stably free. Hence $A \neq R^{[2]}$.
		
		\medskip
		\noindent
		(v)  Suppose there exists a coordinate $V$ in $R[X,Y,Z]$ such that $V \in A$. Now note that $\deg_{D}(X)=\deg_{D}(Y)=1$ and $\deg_{D}(Z)=d+2$. As $D$ is homogeneous, we can assume that $V$ is linear. 
		Since $V$ is a coordinate in $R[X,Y,Z]$, it is irreducible and hence it follows that $A_{1+w_2}=R_{1+w_2}[X_2,F]=R_{1+w_2}[V,F]$ and $A_{1-w_2}=R_{1-w_2}[X_1,F]=R_{1-w_2}[V,F]$.
		Therefore, by local-global principle, we have 
		$ A= R[V,F]=R^{[2]}$ which contradicts (iv). Therefore, $\ker(D)$ can not contain a coordinate in $R[X,Y,Z]$, and 
		hence $rank(D)=3$. Now since $\ker(S^{-1}D)= 
		S^{-1}A=S^{-1}R[X_1, F]$, and $X_1$ is a coordinate in $S^{-1}R[X,Y,Z]$, we have $rank(S^{-1}D)=2$. 
		Therefore, it shows that \thref{lpid} does not hold without the condition that $R$ is a PID. \qed
%
		}
\end{ex}

We now answer Question 1 for PIDs.

\begin{thm}\thlabel{rpid}
	Let $R$ be a {\rm PID} containing $\mathbb{Q}$, $D$ a homogeneous locally nilpotent $R$-derivation on $R[X,Y,Z]$ with respect to the weights $(1,1,1)$. If $\deg(D) \leqslant 3$, then
	$rank(D)<3$ and $\ker(D)=R^{[2]}$.
\end{thm}
\begin{proof}
	Let $S$, $K$, $S^{-1}D$ be as in \thref{lpid}. Since $\deg(S^{-1}D)=\deg(D) \leqslant 3$, by \thref{thma} we have $rank(S^{-1}D)<3$. Therefore, by \thref{lpid} it follows that $rank(D)<3$, and hence $\ker(D)=R^{[2]}$ by \thref{aeh}.   
\end{proof}

As an immediate consequence of the above result we now prove Theorem A. It provides a family of special kind of lnds on $k^{[4]}$, whose kernels are $k^{[3]}$.
\begin{cor}\thlabel{cpid}
Let $k$ be a field of characteristic zero and $E$ be a homogeneous locally nilpotent derivation on $k[X_{1},\dots,X_{4}]$ of degree at most $3$ with respect to the weights $(0,1,1,1)$ with $EX_{1}=0$. Then $\ker(E)=k^{[3]}$ and $rank(E)<3$.
\end{cor}

\begin{proof}
 Let $R=k[X_{1}]$. Then the result follows from \thref{rpid}.  
\end{proof}

 We now investigate Question 1 over Dedekind domains and prove Theorem B. For any Dedekind domain $R$ and a homogeneous $R$-lnd $D$ on $R[X,Y,Z]$ of $\deg(D) \leqslant 3$, the following theorem gives a sufficient condition for $rank(D)<3$.

\begin{thm}\thlabel{dd}
	Let $R$ be a Dedekind domain containing $\mathbb{Q}$ and $D$ be a homogeneous locally nilpotent $R$-derivation on $R[X,Y,Z]$ such that  $\deg(D) \leqslant 3$. 
	Then the following hold.
	
	\smallskip
	\noindent
	 {\rm (i)}
	 $\ker(D)$ is generated by at most $3$ elements.
	 
	 \smallskip
	 \noindent
	 {\rm (ii)}
	 If there exists a minimal system of homogeneous generators for $\ker(D)$ among which exactly one term is linear, then that is a coordinate in $R[X,Y,Z]$ and hence $rank(D)<3$.
\end{thm}
\begin{proof}
	(i) Let $A=\ker(D)$. Since $R$ is a Dedekind domain, $R_{\m}$ is a DVR for every maximal ideal $\m$ of $R$. 
	Note that $D$ can be extended to a homogeneous $R_{\m}$-lnd $D_{\m}$ on $R_{\m}[X,Y,Z]$. 
	By \thref{rpid},
	$A_{\m}=\ker(D_{\m})=R_{\m}^{[2]}.$
	Therefore, by \thref{bcw}, there exists a projective $R$-module $P$ of rank 2 such that
	$ A \cong {\rm{Sym}}_{R}(P).$ 
	Since $rank(P)> \dim(R)$, by \thref{s}, we have $P \cong I \oplus R$ where $I$ is rank 1 projective $R$-module and hence it is a fractional ideal of $R$. As $R$ is a Dedekind domain, the minimum generating set of $I$ can contain at most two elements. Now as $A={\rm{Sym}}_{R}(I \oplus R)$ the result follows.
	
	\smallskip
	\noindent
	(ii) Let	$V: =t_{1}X+t_{2}Y+t_{3}Z$
	be the only linear term among a minimal system of homogeneous generators for $\ker(D)$. If possible, suppose there exists a maximal ideal $\m$ of $R$ such that 
	$(t_{1}, t_{2}, t_{3})R \subseteq \m.$
	Let $\m R_{\m}=(t)R_{\m}$, for some $t\in R$. Then in $R_{\m}[X,Y,Z]$ 	
	$$
	V=(t_{1}X+t_{2}Y+t_{3}Z)=t^{d_{1}}\frac{r_{1}}{s_{1}}X+t^{d_{2}} \frac{r_{2}}{s_{2}}Y+t^{d_{3}}\frac{r_{3}}{s_{3}}Z
	$$
	where $r_{i}, s_{i} \notin \m$ for $i=1,2,3$ and without loss of generality we can assume $d_{1} \leqslant d_{2} \leqslant d_{3}$. Now 
	$$
	s_{1}s_{2}s_{3}V = t^{d_{1}}(r_{1}s_{2}s_{3}X+t^{d_2-d_{1}}r_{2}s_{1}s_{3}Y+t^{d_{3}-d_{1}}r_{3}s_{1}s_{2}Z)\in \ker(D).
	$$
	Therefore,
	$
	U:=(r_{1}s_{2}s_{3}X+t^{d_2-d_{1}}r_{2}s_{1}s_{3}Y+t^{d_{3}-d_{1}}r_{3}s_{1}s_{2}Z) \in \ker(D).
	$
	As $V$ is the only linear element among that system generators of the kernel, there exists $\alpha \in R$ such that 
	$U=\alpha V$, which implies that
	$ r_{1}s_{2}s_{3} = \alpha t_{1} \in \m$. 
	But this is a contadiction. Therefore, $(t_{1},t_{2},t_{3})R=R$. Since $R$ is Dedekind domain, the unimodular row $(t_{1}, t_{2}, t_{3})$ is completable (cf.  \thref{uni}). Hence $V$ is a coordinate in $R[X,Y,Z]$ and $rank(D)<3$.
\end{proof}

\begin{rem}\thlabel{exa 2}
	\em In \thref{r1}, we see that the condition on $\ker(D)$ as in \thref{dd}(ii) is not satisfied and $rank(D)=3$.
	Thus answer to Question 1 is not affirmative in general for Dedekind domains. 
\end{rem}

We now investigate Question 1 for UFDs in higher dimensions $(\geqslant 2)$. The following example gives negative answer to Question 1 for $R=k^{[n]}$, $n \geqslant 2$. The degree of the derivation in the following example is an arbitrary non-negative integer.  

\begin{ex}\thlabel{exa4}
	\em{Let $R=k[t_1,t_2]^{[n-2]}=k^{[n]}$, where $n \geqslant 2$, $B=R[X,Y,Z]$ and $U=t_1X+t_2Y$. Consider the following homogeneous $R$-lnd of degree $d (\geqslant 0)$ on $B$ as follows:
		\begin{align*}
			&DX=t_2U^{d+1},\\
			&DY=-t_1U^{d+1},\\
			&DZ=Y^{d+1}. 
		\end{align*}
		Note that $U \in \ker(D)$ but $U$ is not a coordinate in $B$. Also $\deg_D(X)=\deg_D(Y)=1$ and $\deg_D(Z)=d+2$.
		
		If $rank(D)<3$, then there exists a coordinate $V$ in $B$
		such that $V=aX+bY+cZ$ and $V \in \ker(D)$. If $c \neq 0$, clearly $\deg_D V=d+2$. Hence $V=aX+bY$. Since $DV=0$, $at_2=bt_1$. As $R$ is a UFD and $\gcd(t_1,t_2)=1$, it follows that $t_1 \mid a$ and $t_2 \mid b$. Thus, $V=rU$ for some $r \in R$ and $(a,b)R \subsetneq R$, that means $V$ can not be a coordinate in $B$. Therefore, $rank(D)=3$.	}
\end{ex}

Now in view of \thref{rpid} the following question arises. 

\medskip

\noindent
\textbf{Question 2:}
Let $R$ be a PID containing a field of characteristic zero, and $D$ be a homogeneous locally nilpotent $R$-derivation on $R[X,Y,Z]$ such that $\deg(D) > 3$. Then what information can we get about $rank(D)$ and $\ker(D)$?

\smallskip
\noindent
When $R$ is a field, Freudenburg constructed a family of homogeneous lnds on $k[X,Y,Z]$ of rank 3 and degree $>3$ (cf. In \cite[Theorem 4]{GF}). We modify this example of Freudenburg to construct a homogeneous lnd $D$ on $R[X,Y,Z]$ of $\deg (D)>3$ and $rank(D)=3$ over a PID $R$. We further show that $\ker(D) \neq R^{[2]}$ unlike the case over fields.
 We first record the following lemma.

\begin{lem}\thlabel{ex1}
Let $R,S$ be integral domains such that $R \subseteq S$. Suppose there exists $a \in R$ such that $R_{a}=S_{a}$ and $aS \cap R=aR$. Then $R=S$.
\end{lem}

\begin{ex}\thlabel{exa 1}
{\em
	Let $k$ be a field of characteristic zero, $R=k[t]$ and $B=R[X,Y,Z]$. Consider the following polynomials in $R[X,Y,Z]$ which are homogeneous with respect to the standard weights $(1,1,1)$.
	\begin{align*}
		&F=X(tZ+X)-t^2Y^2,\\
		&G=(tZ+X)F^2+2tX^2YF+X^5,\\
		&P=tYF+X^3.
	\end{align*}
	\noindent
	Now with respect to the standard weights $(1,1,1)$ we define a homogeneous $R$-derivation $D=\Delta_{(F,G)}$ on $B$. Note that
	\begin{align*}
		&DX=-2t^2FP,\\
		&DY=t(6X^2P-G),\\
		&DZ=2X(5t^2YP+tF^2)+2tFP.
	\end{align*}
Let $U=X, V=tY, W=tZ+X.$ Then $k(t)[X,Y,Z]=k(t)[U,V,W]$ and 
\begin{align*}
	& F=UW-V^2,\\
	& G=WF^2+2U^2VF+U^5.
\end{align*}
Let $B_{1}=k(t)[U,V,W]$. Then $D$ induces a homogeneous derivation $D_{1}:=S^{-1}D$ on $B_1$ with respect to the standard grading $(1,1,1)$, where $S=R\setminus\{0\}$. By \cite[Theorem 4]{GF}, \cite[Theorem 5.23]{GFB}, $D_1$ is an lnd such that $\ker(D_{1})=S^{-1}C=k(t)[F,G]$. Now since $B$ is an integral domain, it follows that $D$ is an $R$-lnd on $B$. 
Further note that $rank(D_1)=3$ (cf. \cite[Theorem 4]{GF}, \cite[Theorem 5.23]{GFB}). Since $rank(D) \geqslant rank(D_1)$, it follows that $rank(D)=3$.  

We now show that $A:=\ker(D) \neq R^{[2]}$.
	We first calculate the kernel of $D$. We observe that
	$
	F=X^2+tF_{1}
	$ and  
	$
	G=2X^5+tG_{1},
	$
	where $$F_{1}=XZ-tY^2$$ and
	$$G_{1}=X^4Z+2tX^2F_{1}Z+t^2F_{1}^2Z+2X^3F_{1}+tXF_{1}^2+ 2X^2YF.$$
	Therefore, $F,G,t$ satisfies
	\begin{equation}\label{19}
		G^2-4F^5=tH
	\end{equation}
	where
	$$
	H=4X^5G_{1}+tG_{1}^2-20X^8F_{1}-40tX^6F_{1}^2-40t^2X^4F_{1}^3-20t^3X^2F_{1}^4-4t^4F_{1}^5 .
	$$
	
	By \eqref{19}, we have $H \in A$. Consider the subring $C=k[t,F,G,H] \subseteq A$. We will show that $C=A$. By \thref{ex1}, it is enough to show the following:
	
	\begin{itemize}
		\item [(i)]  $C_{t}=A_{t}$.
		\item [(ii)] $tC=C \cap tA, \text{ i.e., the natural map}\, \frac{C}{tC} \rightarrow \frac{A}{tA}$ is injective. 
	\end{itemize}
	
	\medskip
	\noindent
	(i) We have $C=R[F,G,H]=k[t,F,G,H]$. 
	Now $D$ will induce $D_{t}\in \LND(B_{t})$ such that $A_{t}=\ker(D_{t})$ and 
	$C_{t} \subseteq A_{t} \subset B_{t}=R_{t}[X,Y,Z]$ holds.

	We now show that $A_{t}=C_{t}$.
	Note that $C_{t}=k[t,t^{-1}][F,G]$ and $\td_{k}(C_{t})=\td_{k}(A_{t})=3$. Therefore, it is enough to show that $C_{t}$ is algebraically closed in $B_t$.
	  Now since $S^{-1}C$ is algebraically closed in $B_{1}$ and $S^{-1}C \cap B_t=C_t$, we are done.

	
	
	
	\smallskip
	\noindent 
	(ii) Let $L,M,N$ be algebraically independent elements over $R$. Consider the $R$-algebra epimorphism 
	$$\psi: R[L,M,N] \rightarrow C=R[F,G,H]$$ 
	such that 
	$\psi(L)=F, \psi(M)=G, \psi(N)=H$. 
	Now $(tN-M^2+4L^5) \subseteq \ker(\psi)$ and $(tN-M^2+4L^5)$ is a height 1 prime ideal of $R[L,M,N]$.
	Since $\td_{k} \frac{R[L,M,N]}{(tN-M^2+4L^5)} 
	=\td_{k}(C)$, we have $\ker(\psi)=(tN-M^2+4L^5)$. Therefore, 
	$$C \cong  \frac{R[L,M,N]}{(tN-M^2+4L^5)}=\frac{k[t,L,M,N]}{(tN-M^2+4L^5)}.$$ 
	Clearly $C$ is not regular and hence $C \neq R^{[2]}$.
	We now show that $C/tC \hookrightarrow A/tA$.
	Since $A$ is factorially closed subring of $B$, we have $A/tA \hookrightarrow B/tB$. Therefore, it is enough to show the map 
	$
	\psi_{1}: C/tC \rightarrow  B/tB
	$
	is injective. 
	Since $C/tC \cong \frac{k[L,M,N]}{(M^2-4L^5)}$, it is an integral domain and $\td_{k}(C/tC)=2$. Now that
	$$
	\psi_1(C/tC)=k[\psi_{1}{(F)},\psi_{1}{(G)}, \psi_{1}{(H)}]=k[X^2,X^5,X^9(Y-Z)].
	$$
	Hence $\td_{k}(\psi_1(C/tC))=2$. Thus, $\psi_{1}$ must be injective. 
	
	Thus $A=C \neq R^{[2]}$. \qed
%


}	

\end{ex}

 \section{Appendix}

 In this section we give an alternative proof of \thref{thm1} using \thref{li} proved below. The following result improvises \thref{prop1}.

 \begin{thm}\thlabel{li}
	Let $D$ be a homogeneous locally nilpotent derivation with respect to the standard weights $(1,1,1)$ on $B:=k[X,Y,Z]$ such that $rank(D) >1$. Then there exists a linear system of coordinates $\{L_{1}, L_{2}, L_{3}\}$ in $k[X,Y,Z]$ such that $$\deg_{D}(L_{1})< \deg_{D}(L_{2})< \deg_{D}(L_{3}).$$
\end{thm}

\begin{proof}
	We first assume that the field $k$ is algebraically closed.
	By \thref{zu},  $A~ (:=\ker(D))=k[F,G]$ with $\deg(F) \leqslant \deg(G)$, where $F,G$ are homogeneous polynomials with respect to the standard weights $(1,1,1)$. 
	If possible, suppose all the linear terms in $X,Y,Z$ have the same $\mu$-value. Therefore, 
	$$
	\mu(X)=\mu(Y)=\mu(Z)=n.
	$$ 
	Since $D \neq 0$, $n >0$. Since $D$ is homogeneous with respect to the weights $(1,1,1)$, it is also homogeneous with respect to the weights $\overline{\mu}=(n,n,n)$ (here $\overline{\mu}$ gives the same grading on $k[X,Y,Z]$ as defined in \eqref{m}).
	If $\phi$ is the map as in \eqref{phi}, then
	by \thref{wa}, we obtain that there exists 
	$\alpha \in k^{*}$ such that $X^n-\alpha Y^n \in   \ker(\phi)$. 
	Since $k=\overline{k}$ and $\ker(\phi) \in \Spec(B)$, there exists $c \in k^*$ such that $ X+cY \in \ker(\phi)$ and hence
	$$
	\mu(X+cY) < \overline{\mu}(X+cY)=n=\mu(X).
	$$ 
	This inequality contradicts our assumption, and hence all the linear terms can not have the same $\mu$-value.
	Therefore, there exists a linear system of coordinates 
	$\{V_{1}, V_{2}, V_{3}\}$ such that one of the following occurs:
	\begin{enumerate}
		\item[\rm(i)]  $\mu(V_{1})=\mu(V_{2})< \mu(V_{3})$
		\item[\rm(ii)] $\mu(V_{1})<\mu(V_{2})=\mu(V_{3})$
		\item[\rm(iii)] $\mu(V_{1})<\mu(V_{2})<\mu(V_{3})$
	\end{enumerate}
	Without loss of generality we rename $\{V_{1},V_{2},V_{3}\}$ by $\{X,Y,Z\}$.
	If (iii) happens, then we are done. Now we investigate (i) and (ii). 
	
	\smallskip
	\noindent
	Suppose (i) holds, i.e.,
	\begin{equation}\label{i}
		\mu(X)=\mu(Y)<\mu(Z).
	\end{equation} 
	We now show that there exist $\alpha, \beta \in k^{*}$, such that $\mu(\alpha X+ \beta Y)<\mu(X)$. 
	If possible, suppose for every  $\alpha, \beta \in k^{*}$
	$$
	\mu(\alpha X + \beta Y)= \mu(X)=\mu(Y)< \mu(Z) .
	$$
	Let 
	$$
	F=f_{m}(X,Y)+ f_{m-1}(X,Y)Z+\dots+f_{0}(X,Y)Z^m,
	$$ 
	such that for every integer $i$, $f_{i}(X,Y)$ is a homogeneous polynomial of degree $i$. Since $k$ is algebraically closed, every $f_{i}$ can be written as product of $i$ linear polynomials in $X,Y$. Since $\mu(X)=\mu(Y)<\mu(Z)$, we have
	$
	\mu(f_{m-i}(X,Y)Z^i)< \mu(f_{m-i-1}(X,Y)Z^{i+1}).
	$ 
	Therefore, there exists $j$ such that 
	$
	\mu(f_{m-j}(X,Y)Z^j)=\mu(F)=0.
	$
	If $j \geqslant 1$, then $\mu(Z)=0$ and by \eqref{i}, $\mu(X)=\mu(Y)=0$,
	which contradicts that $D \neq 0$.
	If $j=0$, then $\mu(X)=\mu(Y)=0$ which contradicts $rank(D)>1$. Therefore there exist $\alpha, \beta \in k^{*}$ such that $\mu(\alpha X + \beta Y) < \mu(X)=\mu(Y)$ and hence we are reduced to (iii). 
	
	\smallskip
	\noindent
	Now suppose (ii) holds, i.e., 
	\begin{equation}\label{ii}
		\mu(X)<\mu(Y)=\mu(Z). 
	\end{equation}
	We now show there exist $\alpha, \beta \in k^{*}$ such that 
	$\mu(\alpha Y+ \beta Z)<\mu(Y)=\mu(Z)$. 
	If possible, suppose  
	$$
	\mu(\alpha Y+ \beta Z)=\mu(Y)=\mu(Z),
	$$
	for all $\alpha, \beta \in k^{*}$.
	Since $rank(D)>1$, we must have $\deg(G) \geqslant 2$. Let 
	$$
	G= g_{n}(Y,Z)+g_{n-1}(Y,Z)X+\dots+g_{0}(Y,Z)X^n,
	$$
	where $g_{i}$ is a homogeneous polynomial of degree $i$, for every $0 \leqslant i \leqslant n$.
	Now we have $\mu(g_{i}(Y,Z)X^{n-i})> \mu(g_{i-1}(Y,Z)X^{n-i+1})$ and since $G$ is irreducible, $g_{n}(Y,Z) \neq 0$. 
	Therefore, $\mu(G)=n\mu(Y)=0$. But then by \eqref{ii}, $\mu(X)=\mu(Y)=\mu(Z)=0$, which contradicts that $D \neq 0$. Hence there exist $\alpha, \beta \in k^{*}$ such that $\mu(\alpha Y+ \beta Z)<\mu(Y)=\mu(Z)$.
	Now if $\mu(X)$, $\mu(\alpha Y+ \beta Z)$, $\mu(Y)$ are distinct, then we are reduced to (iii) and 
	if $\mu(X)= \mu(\alpha Y+ \beta Z)< \mu(Y),$ then we are reduced to (i) and hence we are done.
	
	\medskip
	
	We now assume that  $k$ is not algebraically closed and $\overline{k}$ be an algebraic closure of $k$. Then $D$ extends to $\overline{D}=D \otimes_{k} \overline{k} \in \LND(\overline{k}[X,Y,Z])$.\\ Since $D$ is homogeneous, so is $\overline{D}$. Therefore, there exist $L_{1}, L_{2}, L_{3} \in \overline{k}[X,Y,Z]$ linear in $X,Y,Z$ such that 
	$$
	\deg_{\overline{D}}(L_{1})< \deg_{\overline{D}}(L_{2})< \deg_{\overline{D}}(L_{3}).  
	$$
	Let $L_i= \alpha_iX+\beta_iY+\gamma_iZ$ for some $\alpha_i, \beta_i,\gamma_i \in \overline{k}$ and $1 \leqslant i \leqslant 3$.
	Let $k^{\prime}=k(\alpha_i, \beta_i, \gamma_i \mid 1 \leqslant i \leqslant 3)$, and let $\{b_{1}, b_{2}, \dots, b_{r}\}$ be a $k$-basis for $k^{\prime}$. 
	Therefore, we can write
	$$
	L_{i}= \sum_{j=1}^{r} b_{j} V_{ij},
	$$ 
	where 
	$i=1,2,3$ and $V_{ij} \in k[X,Y,Z]$ are linear in $X,Y,Z$. 
	Since $b_{j}$'s are linearly independent over $k$, we must have some $V_{ij_{i}}$ such that 
	$ 
	\mu(V_{ij_{i}})= \deg_{\overline{D}}(L_{i}), 
	$ 
	for every $i$.
	Therefore we have, 
	$$
	\mu(V_{1j_{1}})< \mu(V_{2j_{2}})< \mu(V_{3j_{3}}) 
	$$
	Since $\mu$-values of $V_{1j_{1}},V_{2j_{2}},V_{3j_{3}}$ are distinct, $\{V_{1j_{1}},V_{2j_{2}},V_{3j_{3}}\}$ is an algebraically independent set of linear terms in $k[X,Y,Z]$. Therefore, this is a system of coordinates in $k[X,Y,Z]$ and hence we are done.
\end{proof}

 We now note an easy lemma as follows.
 \begin{lem}\thlabel{irr}
 	Let $B$ be an affine $k$-algebra which is a UFD and $D$ a non-trivial locally nilpotent derivation on $B$. 
 	Let $\overline{k}$ be an algebraic closure of 
 	$k$, $\overline{B}:=B \otimes_{k} \overline{k}$ an integral domain and $\overline{D}$ 
 	be the natural extension of $D$ on $\overline{B}$.
 	If $D$ is irreducible then so is $\overline{D}$.
 \end{lem}
 
 \begin{proof}
 	
 	Let $J=(DB)$ and $\overline{J}=(\overline{D} (\overline{B}))$. Clearly, $\overline{J}=J \overline{B}$. Suppose, if possible, $D$ is irreducible but $\overline{D}$ is not. 
 	Then there exist $\overline{b} \in \overline{B}$ and a prime 
 	ideal $\p$ of $\overline{B}$ such that $\hgt \p =1$ and 
 	$\overline{J} \subseteq (\overline{b}) \subseteq \p$; and 
 	hence $J \subseteq \p \cap B= {\p}_1$. Since $B \subseteq \overline{B}$ is a flat extension, it satisfies the going down property (cf. \cite[5.D, Theorem 4]{mat}), and hence $\hgt {\p}_1=1$. Now, since $B$ is a UFD, ${\p}_1$ is principal, which contradicts that $D$ is irreducible. Hence the result follows. 
 \end{proof}
 
 Before proceeding further, we note the following.
 Let $D$ be an irreducible homogeneous lnd of degree $d$ on $k[U,V,W]$. Let $A= \ker(D)=k[F,G],$ where $F, G$ are irreducible homogeneous polynomials of degree $p$ and $q$ respectively. By \thref{dai}, $D=\Delta_{(F,G)}$ and hence it follows that $d=p+q-3$. Such an lnd $D$ is called an irreducible homogeneous lnd of type-$(p,q)$. 
 The following result has been proved in \cite[Theorem 4]{GF} over any algebraically closed field of characteristic zero, and as a consequence \thref{thm1} follows (\cite[Corollary 2]{GF}). We now give an alternative proof below over an arbitrary field of characteristic zero, using \thref{li}.

 \begin{thm}\thlabel{ra}
 	Let $B=k[U,V,W]$. Then there is no irreducible homogeneous locally nilpotent derivation on $B$ of type-$(2,d+1)$ for $d=1,2,3$, and there is no homogeneous locally nilpotent derivation of type-$(3,3)$.
 \end{thm}
 
 \begin{proof}
 	Let $D$ be an irreducible homogeneous lnd on $B$ and $\overline{k}$ be an algebraic closure of $k$. Then $D$ extends to a homogeneous lnd  $\overline{D}=D\otimes_{k}\overline{k}$ on $B \otimes_{k} \overline{k}= \overline{k}[U,V,W]$. 
 	Note that $\deg(D)=\deg(\overline{D})$, and
 	since $D$ is irreducible, by \thref{irr}, $\overline{D}$ is so. 
 	If $\ker(D)=k[F,G]$, then $\ker(\overline{D})=\overline{k}[F,G]=\overline{k}^{[2]}$. By \thref{dai}, $\overline{D}= \Delta_{(F,G)}$. Therefore, $D$ and $\overline{D}$ have the same type. So without loss of generality we may assume that $k=\overline{k}$ and $D=\overline{D}$.
 	
 	Let $D$ be an lnd of type-$(2,d+1)$ for $1 \leqslant d \leqslant 3$, i.e., $\deg(F)=2$ and $\deg(G)=d+1$. 
 	Note that $rank(D)=3$ as $d \geqslant 1$. 
 	By \thref{li}, there exists a system of coordinates $\{X,Y,Z\}$ linear in $U,V,W$ such that 
 	\begin{equation}\label{dx}
 	\deg_{D}(X)< \deg_{D}(Y)< \deg_{D}(Z).
 \end{equation}
 	Suppose $F=aX^2+bY^2+cZ^2+eXY+fYZ+gZX$ where $a,b,c,e,f,g \in k$.
 	
 	If $c\neq 0$, then $\deg_{D}(F)=\deg_{D}(Z^2)=0$. That means $Z \in \ker(D)$. By \eqref{dx}, this contradicts that $D \neq 0$. Therefore, $c=0$.
 	
 	If $f\neq 0$, then $\deg_{D}(F)=\deg_{D}(YZ)=0$ and hence $Y,Z\in \ker(D)$, which would again contradict that $D \neq 0$. Therefore, $f=0$. Hence we get 
 	$
 	F= bY^2+X(aX+eY+gZ) ,
 	$
 	where $b \neq 0$ and $g \neq 0$, as $F$ is an irreducible polynomial of degree $2$. Now with respect to the following new system of linear coordinates
 	$$
 	U_{1}=X,~~ U_{2}=\sqrt{b}Y, ~~U_{3}= aX+eY+gZ,
 	$$
 	we have $F=(U_{2}^2+U_{1}U_{3})$ and $\deg_{D}(U_{1})< \deg_{D}(U_{2}) < \deg_{D}(U_{3})$. 
 	Renaming $\{U_{1},U_{2},U_{3}\}$ as $\{X,Y,Z\}$ we write 
 	$
 	F=Y^2+XZ
 	$
 	where  
 	\begin{equation}\label{2}
 		\deg_{D}(X)< \deg_{D}(Y) < \deg_{D}(Z).
 	\end{equation}

 	\smallskip
 	
 	\noindent
 	{ \it \textbf{Case}  $\mathbf{d=1}$}:
 	Now $G$ is a degree $2$ homogeneous polynomial. Suppose 
 	$$
 	G=\lambda F+ Y(c_{1}X+ c_{2}Z) + c_{3}ZX+ c_{4}Z^2+ c_{5}X^2,
 	$$
 	where $\lambda, c_{i} \in k$ for $1 \leqslant i \leqslant 5$.
 	Since we have $\deg_{D}(X)< \deg_{D}(Y) < \deg_{D}(Z)$ and $rank(D)=3$, we have $c_{2}=c_{4}=0$, as otherwise $Z \in \ker(D)$. Therefore,
 	\begin{equation}\label{3}
 		G=\lambda F+X(c_{1}Y+c_{3}Z+c_{5}X).
 	\end{equation}
 	Since $G-\lambda F \in \ker(D)$, we see that $X \in \ker(D)$, which contradicts that $rank(D)=3$. 
 	
 	\smallskip
 	
 	\noindent
 	{\it \textbf{Case}}  $\mathbf{d=2}$:
 	Now $G$ is a degree 3 homogeneous polynomial. Suppose
 	$$
 	G= F(a_{1}X+a_{2}Y+a_{3}Z) + Y(b_{1}X^2+b_{2}XZ+b_{3}Z^2)+ d_{1}X^3+d_{2}X^2Z+d_{3}XZ^2+d_{4}Z^3,
 	$$
 	where $a_{i},b_{i},d_{j} \in k$ for $1 \leqslant i \leqslant 3$ and $ 1 \leqslant j \leqslant 4$. 
 	If $d_{4} \neq 0$, then by (\ref{2}), $\deg_{D}(G)= \deg_{D}(Z^3)=0$, and hence $Z \in \ker(D)$. This would contradict that $D \neq 0$. Therefore, $d_{4}=0$. If $b_{3} \neq 0$, then $\deg_{D}(G)=\deg_{D}(YZ^2)=0$, and hence $Y,Z \in \ker(D)$. This is again a contradiction.
 	 Therefore, $b_{3}=0$.
 	
 	Since $G$ is irreducible, $\overline{G} \neq 0$ in $\overline{B}=\frac{k[X,Y,Z]}{(F)}$. Let $\tilde{D}= D\, (\mod F)$ and hence $\tilde{D} \in \LND(\overline{B})$. Now in $\overline{B}$ we have the following equality
 	$$
 	\overline{G}= \overline{X} (b_{1}\overline{XY}+b_{2}\overline{YZ}+d_{1}\overline{X}^2+d_{2}\overline{XZ}+d_{3}\overline{Z}^2).
 	$$
 	Therefore, $\overline{X} \in \ker(\tilde{D})$ and hence $ DX=F(\alpha_{1}X+\alpha_{2}Y+\alpha_{3}Z)$ for some $(\alpha_{1},\alpha_{2},\alpha_{3}) \in k^3 \setminus \{(0,0,0)\}$, as $D$ is homogeneous lnd of degree 2. By (\ref{2}), we have $\alpha_{2}=\alpha_{3}=0$. But then $DX \subseteq (X)$. It would imply that $\deg_D(DX) \geqslant \deg_D(X)$, which is a contradiction.
 	
 	\smallskip
 	
 	\noindent
 	{\it \textbf{Case}} $\mathbf{d=3}$:
 	Now $G$ is a degree 4 homogeneous polynomial. Suppose
 	$$
 	G=\mu F^2+Fg^{\prime}(X,Y,Z)+Xg^{\prime \prime}(X,Y,Z)+e_{0}YZ^3+e_{1}Z^4,
 	$$
 	where $e_{0},e_{1} \in k$ and $g^{\prime}(X,Y,Z) , g^{\prime \prime}(X,Y,Z)$ are homogeneous polynomials of degree $2$ and $3$ respectively and linear in $Y$.
 	By similar arguments as in case $d=2$, we have $e_0=e_1=0$, otherwise $Z \in \ker(D)$ which contradicts that $D \neq 0$ (by \eqref{2}).
 	Now in $\overline{B}$ we have $\overline{G}= \overline{X} g^{\prime \prime}(\overline{X}, \overline{Y}, \overline{Z})$ and hence $\overline{X} \in \ker(\tilde{D})$. Since $D$ is homogeneous lnd of degree $3$, we have 
 	\begin{equation}\label{4}
 		DX= \mu^{\prime} F^2+ F(a^{\prime}X^2+b^{\prime}XY+c^{\prime}XZ+d^{\prime}YZ+e^{\prime}Z^2)
 	\end{equation} 
 	Using (\ref{2}), we have $d^{\prime}, e^{\prime}=0$. Now if $(a^{\prime}X^2+b^{\prime}XY+c^{\prime}XZ) \neq 0$, then $$
 	\deg_{D}(DX)=\deg_{D}(a^{\prime}X^2+b^{\prime}XY+c^{\prime}XZ) \geqslant \deg_{D}(X)
 	$$
 	which is not possible. Therefore, $a^{\prime}=b^{\prime}=c^{\prime}=0$ and hence $DX=\mu^{\prime}F^2$.
 	
 	By \thref{localslice}, we have 
 	$$
 	k[F,G]_{DX}[X]=k[X,Y,Z]_{DX}
 	$$
 	and hence
 	$$
 	k[F,F^{-1},G,X]=k[X,Y,Z,F^{-1}].
 	$$
 	Hence 
 	\begin{equation}\label{fgen}
 		F^nZ=P(F,G,X)
 	\end{equation} for some polynomial $P \in k^{[3]}$ and $n \in \mathbb{N} \cup \{0\}$. Evaluating \eqref{fgen} at $X=0,$ we get 
 	\begin{equation}\label{fgen1}
 		Y^{2n}Z=P(F(0,Y,Z), G(0,Y,Z),0).
 	\end{equation} 
 	But in \eqref{fgen1}, the R.H.S is of even degree where L.H.S is of odd degree. Hence we get a contradiction.
 	
 	Hence we get that $D$ can not be of type-$(2,d+1)$ for $d=1,2,3$.
 	
 	\smallskip
 	Now suppose $D$ be an lnd of type-$(3,3)$. That means $A=\ker(D)=k[F,G]$ where $F$, $G$ both are homogeneous polynomials of degree $3$ and $d=3$. If $H$ is a homogeneous local slice of degree $l$, then $DH$ is a homogeneous element of $A$ of degree $l+3$. As $\ker(D)$ is generated by $F$ and $G$, which are homogeneous of degree 3, we have $l+3=3s$ for some positive integer $s$. Therefore $l=3(s-1)$. But then by \thref{localslice}, we have
 	\begin{equation}\label{F}
 	k[F,G]_{DH}[H]=k[U,V,W]_{DH}.
 	\end{equation}
 	By \eqref{F}, any linear polynomial $L$ in $k[U,V,W]$ can be written as
 	$$
 	L= \frac{p(F,G,H)}{(DH)^n}
 	$$
 	for some $p \in k^{[3]}$ and an integer $n \geqslant 0$.
 	Therefore, 
 	\begin{equation}\label{L}
 		(DH)^n L= p(F,G,H).
 	\end{equation}
 	Since $F,G,H,DH$ all are homogeneous polynomials of degrees divisible by $3$, every monomial in the R.H.S of \eqref{L} has degree divisible by $3$. Whereas, the L.H.S contains monomials of degree $3q+1$ for some integer $q \geqslant 0$, as $L$ is linear. Therefore, we get a contradiction and hence
 	 there is no homogeneous lnd of type-$(3,3)$ on $k[U,V,W]$.
 \end{proof}
 
 The above theorem implies \thref{thm1}, and the details can be found in \cite[Corollary 2]{GF}. However, for the sake of completeness we are giving the proof below. 
 \begin{cor}\thlabel{cra}
 	Let $D$ be a homogeneous locally nilpotent derivation on $k^{[3]}$ of degree not more than $3$.  Then $rank(D)<3$.
 \end{cor}
 \begin{proof}
 	By \thref{dai}, we see that $D=a \Delta_{(F,G)}$, where $\ker(D)=k[F,G]$ and $a \in \ker(D)$. Since $D$ is homogeneous, we have $\Delta_{(F,G)}$ is also homogeneous such that $\deg(\Delta_{(F,G)}) \leqslant \deg(D)$ and $rank(\Delta_{(F,G)})=rank(D)$. Therefore, it is enough to show the result for irreducible homogeneous lnds. Let $d$ be the degree of $D$. Suppose $D$ is of type-$(p,q)$. Then $d=p+q-3$.
 	
 	If $d=0$, then $D$ must be of type-$(1,2)$ and in that case $rank(D)<3$.
 	
 	If $d=1$ or $d=2$, then $D$ can not be of type-$(2,d+1)$ by \thref{ra}. Therefore, it must be of type-$(1,d+2)$ and hence $rank(D)<3$.
 	
 	If $d=3$, again by \thref{ra}, $D$ can not be of type-$(2,4)$ and $(3,3)$. Hence it must be of type-$(1,5)$.
 	
 	Therefore we see for every $d \leqslant 3,$ $\ker(D)$ must contain a linear term and hence a coordinate in $k^{[3]}$. That is $rank(D)<3$. 
 \end{proof}

\section*{Acknowledgement}
The author is thankful to Neena Gupta and Animesh Lahiri for and carefully going through the earlier draft and suggesting improvements.


\begin{thebibliography}{XXX}

\bibitem{AEH} S. S. Abhyankar, P. Eakin and W. Heinzer, {\it On the uniqueness of the coefficient ring in a polynomial ring}, 
J. Algebra {\bf 23} (1972), 310--342.

\bibitem{bass}  H. Bass, {\it $K$-theory and stable algebra,} I.H.E.S. {\bf 22} (1964), 5--60.

\bibitem{bcw} H. Bass, E. H. Connell and D. L. Wright, {\it Locally polynomial algebras are symmetric algebras}, Invent. Math. {\bf 38} (1977), 279--299.

\bibitem{bhd} S.M. Bhatwadekar, D. Daigle, { \it On finite generation of kernels of locally nilpotent R-derivations of R[X,Y,Z]}, J. of Algebra {\bf 322} (2009), 2915--2926.

\bibitem{bgl} S. M. Bhatwadekar, N. Gupta and S. A. Lokhande, {\it Some K-Theoretic properties of the kernel
of a locally nilpotent derivation on $k[X_1, . . .,X_4]$}, Trans.
Amer. Math. Soc. { \bf 369}(2017), 341--363.

\bibitem{nice} N. Dasgupta, N. Gupta, {\it Nice derivations over principal ideal domains}, J. Pure Appl. Algebra {\bf 222} (2018), 4161--4172.

\bibitem{daig} D. Daigle, { \it On some properties of locally nilpotent derivations}, J. Pure Appl. Algebra { \bf 114} (1997),
221–230.


\bibitem{df} D. Daniel, G. Freudenburg, {\it A note on triangular derivations of $k[X_1,X_2,X_3,X_4]$}, Proc. Amer. Math. Soc. {\bf 129} (2001), 657–662.

\bibitem{GF} G. Freudenburg, {\it Actions of $\mathbb{G}_{a}$ on $\mathbb{A}^3$ defined by homogeneous derivations}, J. Pure Appl. Algebra { \bf 126}
(1998), 169--181.

\bibitem{GFB} G. Freudenburg, \textit{Algebraic Theory of Locally Nilpotent Derivations}, SpringerVerlag GmbH Germany (2017).




\bibitem{IR} F. Ischebeck, R. Rao, {\it Ideals and Reality}, Springer-Verlag Berlin Heidelberg (2005).


\bibitem{mat} H. Matsumura, 
{\it Commutative Algebra}, Second edition, 
The Benjamin/ Cummings Publishing Company (1980).

\bibitem{miya} M. Miyanishi,\textit{ Normal affine subalgebras of a polynomial ring}, Algebraic and Topological Theories—
to the memory of Dr. Takehiko Miyata (Kinosaki, 1984), Kinokuniya, Tokyo (1986), 37--51.


\bibitem{ren} R. Rentschler, \textit{Op\'erations du groupe additif sur le plan affine,} C. R. Acad. Sc. Paris {\bf{267}} (1968), 384--387.

\bibitem{su} A. A. Suslin, {\it Locally polynomial rings and symmetric algebras} (Russian), Izv. Akad. Nauk SSSR Ser. Mat. {\bf 41(3)} (1977), 503--515.


\bibitem{1} Z. Wang, \textit{Homogeneization of locally nilpotent derivations and application and an application to $k[X,Y,Z]$}, J. Pure and Applied Algebra {\bf 196} (2005), 323--337.

\bibitem{zur}  V. D. Zurkowski,\textit{ Locally finite derivations}, Rocky Mount. J. Math, to appear.

\end{thebibliography}
\end{document}